\documentclass{amsart}

\usepackage[english]{babel}
\usepackage[utf8]{inputenc}
\usepackage{lmodern}
\usepackage{amsmath,amsfonts,amssymb,amsthm}
\usepackage{mathtools}
\usepackage{enumitem}
\usepackage{ifthen}
\usepackage{url}

\newcommand{\N}{\mathbb{N}}
\newcommand{\Q}{\mathbb{Q}}
\newcommand{\R}{\mathbb{R}}
\newcommand{\Z}{\mathbb{Z}}

\renewcommand{\epsilon}{\varepsilon}

\theoremstyle{plain}
\newtheorem{masterthm}{doNotUse}
\newtheorem{theorem}[masterthm]{Theorem}
\newtheorem{lemma}[masterthm]{Lemma}
\newtheorem{corollary}[masterthm]{Corollary}
\newtheorem{proposition}[masterthm]{Proposition}
\theoremstyle{remark}
\newtheorem{remark}{Remark}

\title[On Pillai's Problem involving Lucas sequences]{On Pillai's Problem involving Lucas sequences of the second kind}



\author[S. Heintze]{Sebastian Heintze}
\address{Sebastian Heintze, Graz University of Technology, Institute of Analysis and Number Theory, Steyrergasse 30/II, A-8010 Graz, Austria}
\email{heintze@math.tugraz.at}

\author[V. Ziegler]{Volker Ziegler}
\address{Volker Ziegler, University of Salzburg, Hellbrunnerstrasse 34/I, A-5020 Salzburg, Austria}
\email{volker.ziegler@plus.ac.at}

\thanks{This work was supported by the Austrian Science Fund (FWF) under the project I4406.}

\begin{document}

\begin{abstract}
 In this paper we consider the Diophantine equation $ V_n - b^m = c $ for given integers $ b,c $ with $ b \geq 2 $, whereas $ V_n $ varies among Lucas-Lehmer sequences of the second kind. We prove under some technical conditions that if the considered equation has at least three solutions $ (n,m) $, then there is an upper bound on the size of the solutions as well as on the size of the coefficients in the characteristic polynomial of $ V_n $.
\end{abstract}

\maketitle

\section{Introduction}

In recent time many authors considered Pillai-type problems involving  linear recurrence sequences. For an overview we refer to \cite{Heintze:2023}. Let us note that these problems are inspired by a result due to S. S. Pillai \cite{Pillai:1936, Pillai:1937} who proved that for given, coprime integers $a$ and $b$ there exists a constant $c_0(a,b)$, depending on $a$ and $b$, such that for any $c>c_0(a,b)$ the equation
\begin{equation}\label{eq:Pillai}
	a^n-b^m=c
\end{equation}
has at most one solution $(n,m)\in \Z_{>0}^2$.

Replacing the powers $a^n$ and $b^m$ by other linear recurrence sequences seems to be a challenging task which was supposedly picked up first in \cite{Ddamulira:2017}, where it was shown that
\begin{equation*}
	F_n-2^m=c
\end{equation*}
has at most one solution $(n,m)\in \Z_{>0}^2$ provided that
\begin{equation*}
	c\notin \{0,1,-1,-3,5,-11, -30, 85\}.
\end{equation*}

More generally, Chim, Pink and Ziegler \cite{Chim:2018} proved that for two fixed linear recurrence sequences $(U_n)_{n\in \N}$, $(V_n)_{n\in \N}$ (with some restrictions) the equation
\begin{equation*}
	U_n - V_m = c
\end{equation*}
has at most one solution $(n,m)\in \Z_{>0}^2$ for all $c\in \Z$,
except if $c$ is in a finite and effectively computable set $\mathcal{C} \subset \Z$ that depends on $(U_n)_{n\in \N}$ and $(V_n)_{n\in \N}$.

In more recent years several attempts were made to obtain uniform results, i.e.\ to allow to vary the recurrence sequences $(U_n)_{n\in \N}$ and $(V_n)_{n\in \N}$ in the result of Chim, Pink and Ziegler \cite{Chim:2018}. In particular, Batte et.\ al.\ \cite{Batte:2022} showed that for all pairs $(p,c)\in \mathbb{P} \times \Z$ with $p$ a prime the Diophantine equation
\begin{equation*}
	F_n-p^m=c
\end{equation*}
has at most four solutions $(n,m)\in \N^2$ with $n\geq 2$. This result was generalized by Heintze et.\ al.\ \cite{Heintze:2023}. They proved (under some technical restrictions) that for a given linear recurrence sequence $(U_n)_{n\in \N}$ there exist an effectively computable bound $ B\geq 2 $ such that for an integer $ b > B $ the Diophantine equation
\begin{equation}\label{eq:Pillai-gen}
	U_n-b^m=c
\end{equation}
has at most two solutions $(n,m)\in \N^2$ with $n\geq N_0$. Here $N_0$ is an effective computable constant depending only on $(U_n)_{n\in \N}$.

In this paper we want to fix $b$ in \eqref{eq:Pillai-gen} and let $(U_n)_{n\in \N}$ vary over a given family of recurrence sequences. In particular we consider the case where $(U_n)_{n\in \N}$ varies over the family of Lucas-Lehmer sequences of the second kind.

\section{Notations and statement of the main results}

In this paper we consider Lucas-Lehmer sequences of the second kind, that is we consider binary recurrence sequences of the form
\begin{equation*}
	V_n(A,B)=V_n=\alpha^n+\beta^n,
\end{equation*}
where $\alpha$ and $\beta$ are the roots of the quadratic polynomial
\begin{equation*}
	X^2-AX+B
\end{equation*}
with $A^2\neq 4B$ and $\gcd(A,B)=1$. In the following we will assume that $V_n$ is non-degenerate. That is we assume $A^2-4B>0$ and $A\neq 0$, since this implies that $\alpha$ and $\beta$ are distinct real numbers with $|\alpha|\neq |\beta|$. We will also assume that $A>0$, which results in $V_n>0$ for all $n\in\N$. Then we consider the Diophantine equation
\begin{equation}\label{eq:PillaiVn}
 V_n-b^m=c,
\end{equation}
where $b,c\in \Z$ with $b>1$ are fixed.

\begin{theorem}\label{th:n1-bound}
 Let $0<\epsilon<1$ be a fixed real number and assume that $|B|<A^{2-\epsilon}$ as well as that the polynomial $X^2-AX+B$ is irreducible. Furthermore, assume that $b\geq 2$ if $c\geq 0$ and $b\geq 3$ if $c<0$. Assume that Equation \eqref{eq:PillaiVn} has three solutions $(n_1,m_1),(n_2,m_2),(n_3,m_3)\in \N^2$ with $n_1>n_2>n_3\geq N_0(\epsilon)$ for the bound $N_0(\epsilon)=\frac{3}{2\epsilon}$. Then there exists an effectively computable constant $C_1=C_1(\epsilon,b)$, depending only on $\epsilon$ and $b$, such that $n_1<C_1$ or $A<32^{1/\epsilon}$. In particular, we can choose
 \begin{equation*}
 	C_1= 4.83\cdot 10^{32}\frac{(\log b)^4}{\epsilon^2}\log\left[5.56\cdot 10^{36} \frac{(\log b)^4}{\epsilon^2}\right]^4.
 \end{equation*}
\end{theorem}

Let us note that the bound $N_0=N_0(\epsilon)$ in Theorem \ref{th:n1-bound} ensures that $V_n$ is strictly increasing for $n\geq N_0$ (see Lemma \ref{lem:increasing} below). Let us also mention that it is essential to exclude the case that $V_{n_2}=V_{n_3}$.

Although we can bound $n_1$ in terms of $b$ and $\epsilon$ our method does not provide upper bounds for $A$ and $|B|$. However, in the case that we are more restrictive in the possible choice of $B$ we can prove also upper bounds for $A$ and $|B|$.

\begin{theorem}\label{th:absolute}
 Let $\kappa\geq 1$ be a fixed real number and assume that $|B|<\kappa A$ as well as that the polynomial $X^2-AX+B$ is irreducible. Furthermore, assume that $b\geq 2$ if $c\geq 0$ and $b\geq 3$ if $c<0$. Assume that Equation \eqref{eq:PillaiVn} has three solutions $(n_1,m_1),(n_2,m_2),(n_3,m_3)\in \N^2$ with $n_1>n_2>n_3\geq 1$. Then there exist effectively computable constants $C_2=C_2(\kappa,b)$ and $C_3=C_3(b)$, depending only on $\kappa$ and $b$, such that $\log A<\max\{C_2,C_3\}$.
 In particular, we can choose
 \begin{equation*}
 	C_2=4.35 \cdot 10^{10}\log (4\kappa) \log b\log\left(5.98 \cdot 10^{11} \log b\right)
 \end{equation*}
 and
 \begin{equation*}
 	C_3= 9.41\cdot 10^{9}\left[\log\left(1.4\cdot 10^{36}(\log b)^4\log\left[2.23\cdot 10^{37}(\log b)^4\right]^4\right) \log b\right]^2.
 \end{equation*}
\end{theorem}

A straight forward application of our bounds yields:

\begin{corollary}\label{cor:ex}
Assume that $|B|<A$. If the Diophantine equation
\begin{equation}\label{eq:Pillai-ex}
	V_n-2^m=c
\end{equation}
with $c\geq 0$ has three solutions $(n_1,m_1),(n_2,m_2),(n_3,m_3)\in \N^2$ with $n_1>n_2>n_3\geq 1$, then either
\begin{itemize}
\item $ A<1024 $ or
\item $n_1<1.2\cdot 10^{40}$ and $\log A<4.48\cdot 10^{13}$.
\end{itemize}
\end{corollary}

Another consequence of Theorem \ref{th:absolute} together with the results of Chim et.\ al.\ \cite{Chim:2018} is that there exist only finitely many Diophantine equations of the form of \eqref{eq:PillaiVn} that admit more than two solutions, provided that $|B|<\kappa A$. The following corollary gives a precise statement.

\begin{corollary}\label{cor:absolute}
 Let $\kappa>0$ be a fixed real number and $b\geq 2$ a fixed integer. Then there exist at most finitely many, effectively computable $9$-tuples $(A,B,c,n_1,m_1,n_2,m_2$, $n_3,m_3)\in \Z^9$ such that
 \begin{itemize}
 \item $A>0$, $|B|< \kappa A$ and $A^2-4B>0$;
 \item $X^2-AX+B$ is irreducible;
 \item if $b=2$, then $c\geq 0$;
 \item $n_1>n_2>n_3\geq 1$ and $m_1,m_2,m_3\geq 1$;
 \item $V_{n_i}(A,B)-b^{m_i}=c$ for $i=1,2,3$.
 \end{itemize}
\end{corollary}

\begin{remark}
 Let us note that for an application of the results of Chim et.\ al.\ \cite{Chim:2018} we have to ensure that $\alpha$ and $b$ are multiplicatively independent. However, $\alpha$ and $b$ are multiplicatively dependent if and only if there exist integers $x,y$ not both zero such that $\alpha^x=b^y$. But $\alpha^x$ cannot be a rational number unless $x=0$ or $\alpha=\sqrt{D}$ for some positive integer $D$. But $\alpha=\sqrt{D}$ would imply $A=0$ which we have excluded. Therefore we can apply their results in our situation and obtain the statement of Corollary \ref{cor:absolute}, provided that we have found an upper bound for $A$.
\end{remark}

Let us give a quick outline for the rest of the paper. In the next section we establish several lemmas concerning properties of Lucas sequences $V_n$ under the restrictions that $|B|<A^{2-\epsilon}$ and $|B|<\kappa A$, respectively, that we will frequently use throughout the paper. The main tool for the proofs of the main theorems are lower bounds for linear forms in logarithms of algebraic numbers. In Section \ref{sec:firstbound} we establish bounds for $n_1$, which still depend on $\log \alpha$, following the usual approach (cf. \cite{Chim:2018}). In Section \ref{sec:three-sol}, under the assumption that three solution exist, we obtain a system of inequalities involving linear forms in logarithms which contain $\log \alpha$. Combinig these inequalities we obtain a linear form in logarithms which does no longer contain $\log \alpha$. Thus we obtain that Theorems \ref{th:n1-bound} and \ref{th:absolute} hold or one of the following two cases occurs:
\begin{itemize}
 \item $n_1m_2-n_2m_1=0$;
 \item $m_2-m_3\ll \log n_1$.
\end{itemize}
That each of these cases implies Theorems \ref{th:n1-bound} and \ref{th:absolute} is shown in the subsequent Sections \ref{sec:degenerate-case} and \ref{sec:case-m2-m3}.

\section{Auxiliary results on Lucas sequences}

Let $\alpha$ and $\beta$ be the roots of $X^2-AX+B$. By our assumptions $\alpha$ and $\beta$ are distinct real numbers $\notin \Q$. Throughout the rest of the paper we will assume that $|\alpha|>|\beta|$, which we can do since by our assumptions we have $A\neq 0$. Therefore we obtain
\begin{equation*}
	\alpha=\frac{A+\sqrt{A^2-4B}}2 \qquad \text{ and } \qquad \beta=\frac{A-\sqrt{A^2-4B}}2.
\end{equation*}

First note that $\alpha>1$ can be bounded in terms of $A$ and $\epsilon$, respectively in terms of $A$ and $\kappa$.

\begin{lemma}\label{lem:alpha-A-est}
 Assuming that $|B|<A^{2-\epsilon}$, we have $\frac{A}{2}<\alpha<2A$. Assuming that $|B|<\kappa A$, we have $A<\kappa$ or $\frac{A}{2}<\alpha<2A$.
\end{lemma}

\begin{proof}
 Assume that $A\geq \kappa$. Then both assumptions $|B|<A^{2-\epsilon}$ and $|B|<\kappa A$ imply $|B|<A^2$. Thus we get
 \begin{equation*}
 	\alpha\leq\frac{A+\sqrt{A^2+4|B|}}2< A\frac{1+\sqrt{5}}2 <2A
 \end{equation*}
 and 
 \begin{equation*}
 	\alpha = \frac{A+\sqrt{A^2-4B}}2> \frac{A}2.
 \end{equation*}
\end{proof}

At some point in our proofs of Theorems \ref{th:n1-bound} and \ref{th:absolute} we will use that $\beta$ cannot be to close to $1$. In particular, the following lemma will be needed.

\begin{lemma}\label{lem:beta-1-bound}
 Assume that $A\geq 4$. Then we have
 \begin{equation*}
 	\left|1-|\beta|\right|\geq \frac 2{2A+5}.
 \end{equation*}
\end{lemma}

\begin{proof}
 First, let us note that the function $f(x)=\frac{A-\sqrt{A^2-4x}}2$ is strictly monotonically increasing for all $x<A^2/4$. Moreover, we have
 $f(A-1)=1$ and $f(-A-1)=-1$. By our assumption that $X^2-AX+B$ is irreducible, a choice for $B$ such that $|\beta|=1$ is not admissible. Therefore we have
 \begin{equation*}
 	\left|1-|\beta|\right|\geq \min\{1-f(A-2),f(A)-1,f(-A)+1,-1-f(-A-2)\}.
 \end{equation*}
We compute
\begin{align*}
 1-f(A-2)&=\frac{-(A-2)+\sqrt{A^2-4(A-2)}}2=\frac{4}{2\left((A-2)+\sqrt{A^2-4(A-2)}\right)}\\
 &\geq \frac{2}{A-2+\sqrt{A^2-2A+1}}=\frac{2}{2A-3},
\end{align*}
provided that $A^2-2A+1\geq A^2-4A+8$, which certainly holds for $A\geq 7/2$. Similar computations yield
\begin{equation*}
	f(A)-1\geq \frac{2}{2A-4},\quad f(-A)+1\geq \frac{2}{2A+4},\quad -1-f(-A-2)\geq \frac{2}{2A+5},
\end{equation*}
provided that $A\geq 4$.
\end{proof}

Next, we show that under our assumptions $|\beta|$ is not to large. 

\begin{lemma}\label{lem:bound-alpha-beta-case1}
 Assume that $|B|<A^{2-\epsilon}$. Then we have $\frac{\alpha}{|\beta|}>\frac{1}{2}A^{\epsilon}$ or $A^{\epsilon}< 8$.
\end{lemma}

\begin{proof}
Let us assume that $A^{\epsilon}\geq 8$.
 First we note that
 \begin{align*}
 	\left|\frac{\alpha}{\beta}\right|&=\left|\frac{A+\sqrt{A^2-4B}}{A-\sqrt{A^2-4B}}\right|=\left|\frac{2A^2-4B+2A\sqrt{A^2-4B}}{4B} \right| \\
 	&\geq \frac{2A^2-4|B|+2A\sqrt{A^2-4|B|}}{4|B|}.
 \end{align*}
 Now we consider the function $f(x)=\frac{2A^2-4x+2A\sqrt{A^2-4x}}{4x}$, which is defined for all $x$ with $0<4x\leq A^2$. Then we have 
 \begin{align*}
 f'(x)&=\frac{4x\left(-4+\frac{-4A}{\sqrt{A^2-4x}}\right)-4\left(2A^2-4x+2A\sqrt{A^2-4x}\right)}{16x^2}\\
 &=\frac{-\frac{16xA}{\sqrt{A^2-4x}}-8A^2-8A\sqrt{A^2-4x}}{16x^2}.
 \end{align*}
 Note that it is immediate that $f'(x)<0$ for $x>0$. 
 
 Thus we deduce
 \begin{align*}
 \left| \frac\alpha\beta\right|&>\frac{2A^2-4A^{2-\epsilon}+2A\sqrt{A^2-4A^{2-\epsilon}}}{4A^{2-\epsilon}}\\
 &=\frac{1}{2}A^{\epsilon}-1+\frac{1}{2}A^\epsilon \sqrt{1-4A^{-\epsilon}}\\
 &=A^{\epsilon}\left(\frac{1}{2}-\frac{1}{A^{\epsilon}}+\frac{1}{2}\sqrt{1-4A^{-\epsilon}}\right)>\frac{1}{2}A^\epsilon,
 \end{align*}
 since we assume that $A^\epsilon \geq 8$. 
\end{proof}

\begin{lemma}\label{lem:bound-alpha-beta-case2}
 Assume that $|B|<\kappa A$. Then we have
 $|\beta|<2\kappa$ or $A< 4 \kappa$.
\end{lemma}

\begin{proof}
 We assume that $A\geq 4\kappa$. Moreover, we have
 $$|\beta|=\left|\frac{A-\sqrt{A^2-4B}}{2}\right|\leq \frac{2|B|}{A+\sqrt{A^2-4|B|}}.$$
 Let us now consider the function $f(x)=\frac{2x}{A+\sqrt{A^2-4x}}$. This function is strictly increasing with $x$ since obviously the numerator is strictly increasing with $x$ while the denominator is strictly decreasing with $x$. Therefore we obtain
 \begin{equation*}
 	|\beta|< \frac{2 \kappa A}{A+\sqrt{A^2-4\kappa A}}=\kappa \frac{2}{1+\sqrt{1-\frac{4\kappa}{A}}}\leq 2 \kappa
 \end{equation*}
 since $A\geq 4\kappa$.
\end{proof}

Now let us take a look at the recurrence sequence $ V_n $.

\begin{lemma}\label{lem:increasing}
 Assume that $|B|<A^{2-\epsilon}$ and $ A^{\epsilon} \geq 8 $. Then $V_n$ is strictly increasing for $n\geq \frac{3}{2\epsilon}$.
\end{lemma}

\begin{proof}
 First, we note that
 \begin{equation*}
 	V_{n+1}-V_n=\alpha^{n+1}+\beta^{n+1}-\alpha^n-\beta^n=\alpha^{n}(\alpha-1)+\beta^{n}(\beta-1)>0
 \end{equation*}
 certainly holds if 
 \begin{equation*}
 	\alpha^n(\alpha-1)>|\beta|^n(|\beta|+1),
 \end{equation*}
 respectively
 \begin{equation}\label{eq:incr}
 \left(\frac{\alpha}{|\beta|}\right)^n>\frac{|\beta|+1}{\alpha-1}.
 \end{equation}
 Note that $\frac{\alpha}{|\beta|}>\frac{1}{2}A^\epsilon$, by Lemma \ref{lem:bound-alpha-beta-case1}, and
 \begin{equation*}
 	|\beta|=\frac{|B|}{\alpha}<\frac{2A^{2-\epsilon}}{A}=2 A^{1-\epsilon}.
 \end{equation*}
 Moreover, note that the smallest possible value for $ \alpha $ is $\frac{1+\sqrt{5}}2$. Therefore
 \eqref{eq:incr} is certainly fulfilled if
 \begin{equation*}
 	\left(\frac{1}{2}A^{\epsilon}\right)^n > 6A^{1-\epsilon} > \frac{3A^{1-\epsilon}}{\frac{-1+\sqrt{5}}2} \geq \frac{2A^{1-\epsilon}+1}{\frac{1+\sqrt{5}}2-1}.
 \end{equation*}
 Thus $V_n$ is increasing if
 \begin{equation*}
 	n\left(\epsilon\log A - \log 2\right)>\log 6+(1-\epsilon) \log A
 \end{equation*}
 or equivalently
 \begin{equation*}
 	n>\frac{(1-\epsilon) \log A+\log 6}{\epsilon\log A-\log 2}.
 \end{equation*}
 Since the rational function $f(x)=\frac{ax+b}{cx+d}$ is strictly increasing if $ad-bc>0$, strictly decreasing if $ad-bc<0$, and constant if $ad-bc=0$, we obtain that for
 \begin{equation*}
 	n \geq \frac{3}{2\epsilon} > \frac{\frac{1-\epsilon}{\epsilon} \log 8 + \log 6}{\log 8 - \log 2}
 \end{equation*}
 the Lucas sequence is strictly increasing.
\end{proof}

\begin{lemma}\label{lem:increasing-2}
 Assume that $|B|<\kappa A$ and $ A \geq 4\kappa + 4 $. Then $V_n$ is strictly increasing for $n\geq 0$.
\end{lemma}

\begin{proof}
 By \eqref{eq:incr} we know that $V_n$ is increasing for all $n\geq 0$ if $1>\frac{|\beta|+1}{\alpha-1}$. Since we assume that $\alpha>\frac{A}{2} \geq 2\kappa+2>2$, we have $|\beta|<2\kappa$, using Lemma \ref{lem:bound-alpha-beta-case2}, and therefore
 \begin{equation*}
 	1\geq \frac{2\kappa+1}{\frac{A}{2}-1}>\frac{|\beta|+1}{\alpha-1}.
 \end{equation*}
\end{proof}

\begin{remark}
 Note that assuming $|B|<\kappa A$ and $A\geq \kappa^2$ implies $|B|<A^{3/2}=A^{2-\epsilon}$ with $\epsilon=1/2$. Therefore all results that are proven under the assumption $|B|<A^{2-\epsilon}$ with $\epsilon=1/2$ also hold under the assumption $|B|<\kappa A$ and $A\geq \kappa^2$.
 \end{remark}

In view of this remark and in view of the proofs of the following lemmata, we will assume for the rest of the paper that one of the following two assumptions holds:
 \begin{description}
  \item[\qquad A1] $|B|<A^{2-\epsilon}$, $A^{\epsilon}\geq 32$ and $N_0= \frac{3}{2\epsilon}$,
  \item[\qquad A2] $|B|<\kappa A$, $A\geq \max\{\kappa^2,16\kappa+12,1024\}$ and $N_0=1$.
 \end{description}
 
 \begin{remark}\label{rem:A1-A2}
 Let us note that the bound $A\geq \kappa^2$ and $A\geq 1024$ in Assumption A2 results in the useful fact that Assumption A2 implies Assumption A1 with $\epsilon=\frac{1}{2}$, but with $N_0=1$ instead of $\frac{3}{2\epsilon}$. The assumption $A\geq 16\kappa+12$ is mainly used in the proofs of the Lemmata \ref{lem:growth-Vn} and \ref{lem:n1-m1-relation}.
 \end{remark}
 
In view of the assumptions stated above an important consequence of Lemmata~\ref{lem:increasing} and \ref{lem:increasing-2} is the following:

\begin{corollary}
Let us assume that \eqref{eq:PillaiVn} has two solutions $(n_1,m_1)$ and $(n_2,m_2)$ with $n_1>n_2\geq N_0$. Then $m_1>m_2$.
\end{corollary}

\begin{lemma}\label{lem:growth-Vn}
 Let $n\geq N_0$, then we have $\frac 54 \alpha^n>V_n>\frac 34 \alpha^n$.
\end{lemma}

\begin{proof}
 Assume that A1 holds. Then, by Lemma \ref{lem:bound-alpha-beta-case1}, we have 
 \begin{equation*}
 	\left(\frac{|\beta|}\alpha\right)^n<\left(\frac{2}{A^\epsilon}\right)^n< \frac{1}{4}.
 \end{equation*}
 If A2 holds, by Lemma \ref{lem:bound-alpha-beta-case2}, we have
 \begin{equation*}
 	\left(\frac{|\beta|}\alpha\right)^n<\left(\frac{4\kappa}{A}\right)^n<\frac 14.
 \end{equation*}
 Therefore in any case we get
 \begin{equation*}
 	\frac{3}{4} \alpha^n < \alpha^n\left(1-\left(\frac{|\beta|}{\alpha}\right)^n\right) \leq \alpha^n+\beta^n=V_n \leq \alpha^n\left(1+\left(\frac{|\beta|}{\alpha}\right)^n\right) < \frac{5}{4} \alpha^n.
 \end{equation*}
\end{proof}

Another lemma that will be used frequently is the following:

\begin{lemma}\label{lem:n1-m1-relation}
 Assume that Equation \eqref{eq:PillaiVn} has two solutions $(n_1,m_1)$ and $(n_2,m_2)$ with $n_1>n_2\geq N_0$. Then we have
 \begin{equation*}
 	\frac{5}{2} > \frac{b^{m_1}}{\alpha^{n_1}}>\frac{3}{8}.
 \end{equation*}
\end{lemma}

\begin{proof}
 Assuming the existence of two different solutions implies by an application of Lemma~\ref{lem:growth-Vn} the inequality 
 \begin{equation*}
 	\frac{5}{4} \alpha^{n_1}>V_{n_1}>V_{n_1}-V_{n_2}=b^{m_1}-b^{m_2}= b^{m_1}\left(1-\frac 1{b^{m_1-m_2}}\right)\geq \frac{1}{2} b^{m_1},
 \end{equation*}
 which proves the first inequality.
 
 For the second inequality we apply again Lemma \ref{lem:growth-Vn} to obtain
 \begin{equation*}
 	V_{n_1}-V_{n_2}>\frac{3}{4} \alpha^{n_1}\left(1-\frac{5}{3\alpha^{n_1-n_2}}\right).
 \end{equation*}
 Since we assume in any case that $\alpha>\frac{A}{2}>4$, we get $V_{n_1}-V_{n_2}>
 \frac{3}{8} \alpha^{n_1}$ and thus
 \begin{equation*}
 	\frac{3}{8} \alpha^{n_1}<V_{n_1}-V_{n_2}=b^{m_1}-b^{m_2}< b^{m_1},
 \end{equation*}
 which yields the second inequality.
\end{proof}

Finally, let us remind Carmichael's theorem \cite[Theorem XXIV]{Carmichael:1913}:

\begin{lemma}\label{lem:Carmichael}
  For any $n\neq 1,3$ there exists a prime\footnote{This prime $p$ is called a \emph{primitive divisor}.} $p$ such that $p\mid V_n$, but $p\nmid V_m$ for all $m<n$, except for the case that $n=6$ and $(A,B)=(1,-1)$.
\end{lemma}

This lemma can be used to prove the following result:

\begin{lemma}\label{lem:c0}
 Assume that $c=0$ and $A>1$. Then the Diophantine equation \eqref{eq:PillaiVn} has at most one solution $(n,m)$ with $n\geq 1$.
\end{lemma}

\begin{proof}
 Assume that \eqref{eq:PillaiVn} has two solutions $(n_1,m_1), (n_2,m_2)$ with $n_1>n_2\geq 1$. Then by Carmichael's primitive divisor theorem (Lemma \ref{lem:Carmichael}) we deduce that $n_1=3$ and $n_2=1$. Since $V_1=A$ and $V_3=A^3-3AB$ we obtain the system of equations
 \begin{equation*}
 	A=b^{m_2}, \qquad A^3-3AB=b^{m_1}
 \end{equation*}
 which yields
 \begin{equation*}
 	b^{2m_2}-3B=b^{m_1-m_2}.
 \end{equation*}
 That is $b\mid 3B$. Since we assume that $\gcd (A,B)=1$, we deduce $b\mid 3$, i.e.\ $b=3$. We also conclude that $3\nmid B$ which yields, by considering $3$-adic valuations, that $m_1-m_2=1$ since $m_2=0$ would imply $A=1$. Hence we have $B=3^{2m_2-1}-1$.
 
 Note that we also assume $A^2-4B>0$ which implies that
 \begin{equation*}
 	3^{2m_2}-4 (3^{2m_2-1}-1)= 4-3^{2m_2-1}>0
 \end{equation*}
 holds. But this is only possible for $m_2=1$ and we conclude that $A=3$ and $B=2$. Since $X^2-3X+2=(X-2)(X-1)$ is not irreducible, this is not an admissible case. Therefore there exists at most one solution to \eqref{eq:PillaiVn}.
\end{proof}

\section{Lower bounds for linear forms in logarithms}\label{sec:linforms}

The main tool in proving our main theorems are lower bounds for linear forms of logarithms of algebraic numbers. In particular, we will use Matveev's lower bound proven in \cite{Matveev:2000}. Therefore let $\eta\neq 0$ be an algebraic number of degree $\delta$ and let
\begin{equation*}
	a_{\delta}\left(X-\eta^{(1)}\right) \cdots \left(X-\eta^{(\delta)}\right) \in \Z[X]
\end{equation*}
be the minimal polynomial of $\eta$. Then the absolute logarithmic Weil height is defined by
\begin{equation*}
	h(\eta)=\frac 1\delta \left(\log |a_\delta|+\sum_{i=1}^\delta \max\{0,\log|\eta^{(i)}|\}\right).
\end{equation*}
With this notation the following result due to Matveev \cite{Matveev:2000} holds:

\begin{lemma}[Theorem~2.2 with $r=1$ in
  \cite{Matveev:2000}]\label{lem:matveev}
  Denote by $\eta_1, \ldots, \eta_N$ algebraic numbers, neither $0$ nor $1$,
  by $\log\eta_1, \ldots, \log\eta_N$ determinations of their logarithms,
  by $D$ the degree over $\Q$ of the number field $K =
  \Q(\eta_1,\ldots,\eta_N)$, and by $b_1, \ldots, b_N$ rational
  integers with $b_N\neq 0$. Furthermore let $\kappa=1$ if $K$ is real and $\kappa=2$
  otherwise. For all integers $j$ with $1\leq j\leq N$ choose
  \begin{equation*}
  	A_j\geq \max\{D h(\eta_j), |\log\eta_j|, 0.16\},
  \end{equation*}
  and set
  \begin{equation*}
  	E=\max\left\{\frac{|b_j| A_j}{A_N}\: :\: 1\leq j \leq N \right\}.
  \end{equation*}
  Assume that
  \begin{equation*}
    \Lambda:=b_1\log \eta_1+\cdots+b_N \log \eta_N \neq 0.
  \end{equation*}
  Then
  \begin{equation*}
    \log |\Lambda|
    \geq -C(N,\kappa)\max\{1,N/6\} C_0 W_0 D^2 \Omega
  \end{equation*}
  with
  \begin{gather*}
    \Omega=A_1\cdots A_N, \\
    C(N,\kappa)= \frac {16}{N!\, \kappa} e^N (2N +1+2 \kappa)(N+2)
    (4(N+1))^{N+1} \left( \frac 12 eN\right)^{\kappa}, \\
    C_0= \log\left(e^{4.4N+7}N^{5.5}D^2 \log(eD)\right),
    \quad W_0=\log(1.5eED \log(eD)).
  \end{gather*}
\end{lemma}

In our applications we will be in the situation $N \in \{2,3\}$ and $K=\Q(\alpha)\subseteq \R$, i.e.\ we have $D=2$ and $ \kappa=1 $. In this special case Matveev's lower bounds take the following form:  

\begin{corollary}\label{cor:matveev-spec}
 Let the notations and assumptions of Lemma \ref{lem:matveev} be in force. Furthermore assume $D=2$ and $ \kappa=1 $. Then we have
 \begin{equation}\label{eq:linearform-bound}
 \begin{aligned}
  \log|\Lambda|&\geq - 7.26\cdot 10^{10}\log(13.81E)\Omega &\qquad \text{for $N=3$},\\
  \log|\Lambda|&\geq - 6.7\cdot 10^{8} \log(13.81E)\Omega &\qquad \text{for $N=2$}.\\
 \end{aligned}
 \end{equation}
\end{corollary}

\begin{remark}
 Let us note that the form of $E$ is essential in our proof to obtain an absolute bound for $n_1$ in Theorem \ref{th:n1-bound}. Let us also note that in the case of $N=2$ one could use the results of Laurent \cite{Laurent:2008} to obtain numerically better values but with an $\log(E)^2$ term instead. This would lead to numerically smaller upper bounds for concrete applications of our theorems. However, we refrain from the application of these results to keep our long and technical proof more concise.
\end{remark}

In order to apply Matveev's lower bounds, we provide some height computations. First, we note the following well known properties of the absolute logarithmic height for any $\eta, \gamma \in \overline{\Q}$ and $l \in \Q$
(see for example \cite[Chapter 3]{Zannier:DA} for a detailed reference):   
\begin{align*}
h(\eta \pm \gamma) & \leq h(\eta) + h(\gamma) + \log 2, \\
h(\eta \gamma) & \leq h(\eta) + h(\gamma),\\
h(\eta^l) & = |l| h(\eta).  
\end{align*}
Moreover, note that for a positive integer $b$ we have $h(b)=\log b$ and 
\begin{equation*}
	h(\alpha)=\frac 12 \left(\max\{0,\log \alpha\}+\max\{0,\log |\beta|\}\}\right)\leq \log \alpha.
\end{equation*}
This together with the above mentioned properties yields the following lemma:

\begin{lemma}\label{lem:heightcomp}
 Under the assumptions A1 or A2 and using the above notations from Lemma~\ref{lem:matveev} the following inequalites hold for any $t\in \Z_{>0}$:
 \begin{align*}
  2\log b & \geq \max\{Dh(b),|\log b|,0.16\},\\
  2\log \alpha & \geq \max\{Dh(\alpha),|\log \alpha|,0.16\},\\
  2(t+1)\log b & \geq \max\{Dh(b^t\pm 1),|\log (b^t\pm 1)|,0.16\},\\
  2(t+1)\log \alpha & \geq \max\{Dh(\alpha^t\pm 1),|\log (\alpha^t\pm 1)|,0.16\}.
 \end{align*}
\end{lemma}

One other important aspect in applying Matveev's result (Lemma \ref{lem:matveev}) is that the linear form $\Lambda$ does not vanish. We will resolve this issue with the following lemma:

\begin{lemma}
 Assume that the Diophantine equation \eqref{eq:PillaiVn} has three solutions $(n_1,m_1)$, $(n_2,m_2)$, $(n_3,m_3)\in \N^2$ with $n_1>n_2>n_3>0$. Then we have
 \begin{align*}
 \Lambda_i&:=n_i\log\alpha-m_i\log b\neq 0,& \text{for $i=1,2,3$}\\
 \Lambda'&:=n_2\log\alpha-m_2\log b+\log\left(\frac{\alpha^{n_1-n_2}-1}{b^{m_1-m_2}-1}\right)\neq 0,
 \end{align*}
\end{lemma}

\begin{proof}
 Assume that $\Lambda_i=n_i\log\alpha-m_i\log b=0$ for some $i\in\{1,2,3\}$. But $\Lambda_i=0$ implies $\alpha^{n_i}-b^{m_i}=0$ which results in view of \eqref{eq:PillaiVn} in
 \begin{equation*}
 	\alpha^{n_i}+\beta^{n_i}-b^{m_i}=\beta^{n_i}=c.
 \end{equation*}
 Since $X^2-AX+B$ is irreducible, $\alpha$ and $\beta$ are Galois conjugated. Therefore, by applying the non-trivial automorphism of $K=\Q(\alpha)$ to the equation $\beta^{n_i}=c$, we obtain $\alpha^{n_i}=c$ since $c\in \Q$. But this implies $\beta^{n_i}=\alpha^{n_i}$, hence $|\alpha|=|\beta|$, a contradiction to our assumptions.

 Now, let us assume that
 \begin{equation*}
 	\Lambda'=n_2\log\alpha-m_2\log b+\log\left(\frac{\alpha^{n_1-n_2}-1}{b^{m_1-m_2}-1}\right)= 0.
 \end{equation*}
 This implies $\alpha^{n_1}-\alpha^{n_2}=b^{m_1}-b^{m_2}$ which results in view of \eqref{eq:Pillai-twosol} in $\beta^{n_1}=\beta^{n_2}$. But then $\beta=0$ or $ \beta $ is a root of unity. Both cases contradict our assumption that $X^2-AX+B$ is irreducible and $\beta \in \R$.
\end{proof}

Finally, we want to record three further elementary lemmata that will be helpful. The first lemma is a standard fact from real analysis.

\begin{lemma}\label{lem:log-est}
	If $|x-1|\leq 0.5$, then $|\log x| \leq 2 |x-1|$, and if $|x|\leq 0.5$, then we have
	\begin{equation*}
		\frac{2}{9} x^2 \leq x-\log (1+x) \leq 2 x^2.
	\end{equation*}
\end{lemma}

\begin{proof}
 A direct application of Taylor's theorem with a Cauchy and Lagrange remainder, respectively.
\end{proof}

Next, we want to state another estimate from real analysis:

\begin{lemma}\label{lem:exp-est}
 Let $x,n\in \R$ such that $|2nx|<0.5$ and $n\geq 1$. Then we have
 \begin{equation*}
 	\left|(1+x)^n-1\right|\leq 2.6n|x|.
 \end{equation*}
\end{lemma}

\begin{proof}
 Since $e^y$ is a convex function, we have for $0\leq y <0.5$ that
 \begin{equation*}
 	e^y\leq 1+y \frac{e^{0.5}-1}{\frac 12}\leq 1+1.3 y,
 \end{equation*}
 and for $-0.5<y\leq 0$ we obtain
 \begin{equation*}
 	e^y\leq 1+y\frac{e^{-0.5}-1}{\frac 12}\leq 1-0.79y.
 \end{equation*}
 That is we have $|1-e^y|\leq 1.3 |y|$ for $|y|<0.5$. 
 
 Note that by our assumptions we have $|x|<0.5$ and therefore we obtain by an application of Lemma \ref{lem:log-est} 
 \begin{equation*}
 	\left|(1+x)^n-1\right|=\left|\exp(n\log(1+x))-1\right|\leq |\exp(2nx)-1|\leq 2.6n|x|.
 \end{equation*}
\end{proof}

The third lemma is due to Peth\H{o} and de Weger \cite{Pethoe:1986}.

\begin{lemma}\label{lem:log-sol}
Let $a,b \geq 0$, $h \geq 1$ and $x\in \R$ be the
largest solution of $x=a + b(\log x)^h$. If $b > (e^2/h)^h$, then
\begin{equation*}
	x<2^h\left(a^{1/h}+b^{1/h}\log (h^h b)\right)^h,
\end{equation*}
and if $b\leq (e^2/h)^h$, then
\begin{equation*}
	x\leq 2^h\left(a^{1/h}+2e^2 \right)^h.
\end{equation*}
\end{lemma}

A proof of this lemma can be found in \cite[Appendix B]{Smart:DiGl}.

\section{A lower bound for $|c|$ in terms of $n_1$ and $\alpha$}\label{sec:c-low-bound}

The purpose of this section is to prove a lower bound for $|c|$. In particular we prove the following proposition:

\begin{proposition}\label{prop:c-low-bound}
 Assume that Assumption A1 or A2 holds and assume that Diophantine equation \eqref{eq:PillaiVn} has two solutions $(n_1,m_1)$ and $(n_2,m_2)$ with $n_1>n_2\geq N_0$. Then we have
 \begin{equation*}
 	|c|>\alpha^{n_1-K_0\log(27.62 n_1)}-|\beta|^{n_1}
 \end{equation*}
 with $K_0=2.69 \cdot 10^9 \log b$.
\end{proposition}

\begin{proof}
Let us assume that, in contrary to the content of the proposition,
\begin{equation*}
	\frac{|c|+|\beta|^{n_1}}{\alpha^{n_1}}\leq \alpha^{-K_0\log (27.62 n_1)}<\frac{1}{2}.
\end{equation*}
We consider equation
\begin{equation*}
	\alpha^{n_1}+\beta^{n_1}-b^{m_1}=c
\end{equation*}
and obtain
\begin{equation*}
	\left|\frac{b^{m_1}}{\alpha^{n_1}}-1\right|\leq \frac{|c|+|\beta|^{n_1}}{\alpha^{n_1}}<\frac{1}{2}
\end{equation*}
which yields
\begin{equation*}
	\left|m_1\log b-n_1\log\alpha\right|\leq 2\frac{|c|+|\beta|^{n_1}}{\alpha^{n_1}}.
\end{equation*}

The goal is to apply Matveev's theorem (Lemma \ref{lem:matveev}) with $N=2$. Note that with $\eta_1=b$ and $\eta_2=\alpha$ we choose $A_1=2\log b$ and $A_2=2\log\alpha$, in view of Lemma~\ref{lem:heightcomp}, and obtain
\begin{equation*}
	E=\max\left\{\frac{m_1\log b}{\log \alpha},n_1\right\}.
\end{equation*}
Due to Lemma \ref{lem:n1-m1-relation} we have
\begin{equation*}
	\frac{m_1\log b}{\log \alpha}<n_1+\frac{\log(5/2)}{\log \alpha}<2n_1.
\end{equation*}
Therefore we obtain by Corollary \ref{cor:matveev-spec} that
\begin{align*}
2.68 \cdot 10^9 \log b\log \alpha \log(27.62 n_1) &\geq -\log\left|m_1\log b-n_1\log\alpha\right| \\
&\geq n_1\log \alpha - \log(|c|+|\beta|^{n_1}) -\log 2
\end{align*}
which implies the content of the proposition.
\end{proof}

\section{Bounds for $n_1$ in terms of $\log \alpha$}\label{sec:firstbound}

In this section we will assume that Assumption A1 or A2 holds. However, in the proofs we will mainly consider the case that Assumption A1 holds. Note that this is not a real restriction since Assumption A2 implies Assumption A1 with $\epsilon=\frac 12$ and $N_0=1$ instead of $N_0=\frac{3}{2\epsilon}$ (cf.\ Remark \ref{rem:A1-A2}). Also assume that Diophantine equation~\eqref{eq:PillaiVn} has three solutions $(n_1,m_1)$, $(n_2,m_2)$, $(n_3,m_3)$ with $n_1>n_2>n_3\geq N_0$. In this section we follow the approach of Chim et.\ al.\ \cite{Chim:2018} and prove upper bounds for $n_1$ in terms of $\alpha$. To obtain explicit bounds and to keep track of the dependence on $\log b$ and $\log \alpha$ of the bounds we repeat their proof. This section also delivers the set up for the later sections which provide proofs of our main theorems. Moreover, note that the assumption that three solutions exist, simplifies the proof of Chim et.\ al.\ \cite{Chim:2018}.

The main result of this section is the following statement:

\begin{proposition}\label{prop:n1bound-logalpha}
 Assume that Assumption A1 or A2 holds and that Diophantine equation \eqref{eq:PillaiVn} has three solutions $(n_1,m_1)$, $(n_2,m_2)$, $(n_3,m_3)$ with $n_1>n_2>n_3\geq N_0$. Then we have
 \begin{equation*}
 	n_1<2.58\cdot 10^{22} \frac{\log \alpha}{\epsilon} (\log b)^2\log\left(7.12\cdot 10^{23}\frac{\log \alpha}{\epsilon} (\log b)^2\right)^2,
 \end{equation*}
 where we choose $\epsilon=1/2$ in case that Assumption A2 holds.
\end{proposition}

Since we assume the existence of two solutions, we have
\begin{equation*}
	V_{n_1}-b^{m_1}=c=V_{n_2}-b^{m_2}
\end{equation*}
and therefore obtain
\begin{equation}\label{eq:Pillai-twosol}
	\alpha^{n_1}+\beta^{n_1}-\alpha^{n_2}-\beta^{n_2}=b^{m_1}-b^{m_2}.
\end{equation}
Let us write $\gamma :=\min\{\alpha,\alpha/|\beta|\}$. Note that we have $\gamma>\frac{1}{2} A^\epsilon>\frac{1}{4} \alpha^\epsilon$, by Lemma~\ref{lem:bound-alpha-beta-case1} and Lemma~\ref{lem:alpha-A-est}. With this notation we get the inequality
\begin{equation*}
	\left|\frac{b^{m_1}}{\alpha^{n_1}}-1\right|\leq \alpha^{n_2-n_1}+\frac{b^{m_2}}{\alpha^{n_1}}+\frac{|\beta|^{n_1}+|\beta|^{n_2}}{\alpha^{n_1}}.
\end{equation*}
Note that, depending on whether $|\beta|>1$ or $|\beta|\leq 1$, we have
\begin{equation*}
	\frac{|\beta|^{n_1}+|\beta|^{n_2}}{\alpha^{n_1}}\leq
	\begin{dcases}
		\dfrac{2}{\alpha^{n_1}} \leq 2\gamma^{-n_1} &\quad \text{if $|\beta|\leq 1$} \\
		2 \cdot \dfrac{|\beta|^{n_1}}{\alpha^{n_1}} \leq 2\gamma^{-n_1} &\quad \text{if $|\beta|> 1$}
	\end{dcases}.
\end{equation*}
Therefore, using Lemma \ref{lem:n1-m1-relation}, we obtain
\begin{equation}\label{eq:L1-exp}
\begin{split}
\left|\frac{b^{m_1}}{\alpha^{n_1}}-1\right|&\leq \alpha^{n_2-n_1}+\frac{b^{m_2}}{\alpha^{n_1}}+2\gamma^{-n_1}\\
&\leq 2\max\left\{\frac 72 \alpha^{n_2-n_1},2\gamma^{-n_1}\right\}\\
&\leq 7\max\left\{\alpha^{n_2-n_1},\gamma^{-n_1}\right\}.
\end{split}
\end{equation}

First, let us assume that the maximum in \eqref{eq:L1-exp} is $\gamma^{-n_1}$. Under our assumptions we have $A^{\epsilon}\geq 32$ and $n_1\geq 3$, i.e.\ $7\gamma^{-n_1}<\frac{1}{2}$. Thus taking logarithms and applying Lemma \ref{lem:log-est} yields
\begin{equation*}
	|\underbrace{m_1\log b-n_1\log\alpha}_{=:\Lambda}| \leq 14\gamma^{-n_1}.
\end{equation*}
We apply Matveev's theorem (Lemma \ref{lem:matveev}) with $N=2$. Note that with $\eta_1=b$ and $\eta_2=\alpha$ we choose $A_1=2\log b$ and $A_2=2\log\alpha$, in view of Lemma \ref{lem:heightcomp}, to obtain
\begin{equation*}
	E=\max\left\{\frac{m_1\log b}{\log \alpha},n_1\right\}.
\end{equation*}
Note that, due to Lemma \ref{lem:n1-m1-relation},
\begin{equation*}
	\frac{m_1\log b}{\log \alpha}<n_1+\frac{\log(5/2)}{\log \alpha}<2n_1.
\end{equation*}
Therefore we obtain from Corollary \ref{cor:matveev-spec} that
\begin{align*}
	2.68 \cdot 10^9 \log b\log \alpha \log(27.62 n_1) &\geq -\log|\Lambda| \geq n_1 \log \gamma -\log 14 \\
	&\geq n_1 (\epsilon \log \alpha - \log 4) -\log 14 \\
	&\geq n_1 \frac{\epsilon}{2} \log \alpha -\log 14,
\end{align*}
where for the last inequality we used $ A^{\epsilon} \geq 32 $. Thus we have
\begin{equation*}
	5.38 \cdot 10^9 \frac{\log b}{\epsilon} \log(27.62 n_1)>n_1,
\end{equation*}
which, using Lemma \ref{lem:log-sol}, proves Proposition \ref{prop:n1bound-logalpha} in this case. Note that this also proves, in this specific case, Theorem \ref{th:n1-bound}.

Now we assume that the maximum in \eqref{eq:L1-exp} is $\alpha^{n_2-n_1}$. By our assumptions on $A$ we have $7\alpha^{n_2-n_1}<\frac{1}{2}$
and obtain, by Lemma \ref{lem:log-est},
\begin{equation*}
	|\underbrace{m_1\log b-n_1\log\alpha}_{=:\Lambda}|\leq 14\alpha^{n_2-n_1}.
\end{equation*}
As computed before, an application of Matveev's theorem yields
\begin{equation*}
	2.68 \cdot 10^9 \log b\log \alpha \log(27.62 n_1) \geq -\log|\Lambda|\geq (n_1-n_2) \log \alpha -\log 14
\end{equation*}
and therefore
\begin{equation}\label{eq:n1-n2-bound}
 2.69 \cdot 10^9 \log b \log(27.62 n_1)>n_1-n_2.
\end{equation}

For the rest of the proof of Proposition \ref{prop:n1bound-logalpha} we will assume that \eqref{eq:n1-n2-bound} holds. Since we assume that a third solution exists, the statement of Lemma \ref{lem:n1-m1-relation} also holds for $m_2$ and $n_2$ instead of $m_1$ and $n_1$. In particular we have
\begin{align*}
	m_1 &< n_1 \frac{\log \alpha}{\log b} + \frac{\log \frac{5}{2}}{\log b} \leq n_1 \frac{\log \alpha}{\log b} + 2\log \frac{5}{2} \\
	m_2 &> n_2 \frac{\log \alpha}{\log b} + \frac{\log \frac{3}{8}}{\log b} \geq n_2 \frac{\log \alpha}{\log b} + 2\log \frac{3}{8}
\end{align*}
which yields
\begin{equation}\label{eq:m1-m2-bound}
m_1-m_2< \frac{\log \alpha}{\log b}(n_1-n_2)+\log 49< 2.7 \cdot 10^9 \log \alpha \log(27.62 n_1).
\end{equation}

Let us rewrite Equation \eqref{eq:Pillai-twosol} again to obtain the inequality
\begin{equation*}
	\left|\frac{b^{m_2}}{\alpha^{n_2}} \cdot\frac{b^{m_1-m_2}-1}{\alpha^{n_1-n_2}-1}-1\right|\leq \frac{|\beta|^{n_1}+|\beta|^{n_2}}{\alpha^{n_1}-\alpha^{n_2}} \leq 4\gamma^{-n_1}.
\end{equation*}
As previously noted we have $4\gamma^{-n_1}<\frac{1}{2}$ and therefore we obtain
\begin{equation*}
	\left|\smash{\underbrace{m_2\log b- n_2\log \alpha + \log \left(\frac{b^{m_1-m_2}-1}{\alpha^{n_1-n_2}-1}\right)}_{=:\Lambda'}}
	\vphantom{\left(\frac{b^{m_1-m_2}-1}{\alpha^{n_1-n_2}-1}\right)}\right|
	\vphantom{\underbrace{m_2\log b- n_2\log \alpha + \log \left(\frac{b^{m_1-m_2}-1}{\alpha^{n_1-n_2}-1}\right)}_{=:\Lambda'}}
	\leq 8\gamma^{-n_1}.
\end{equation*}
We aim to apply Matveev's theorem to $\Lambda'$ with $\eta_3=\frac{b^{m_1-m_2}-1}{\alpha^{n_1-n_2}-1}$. Note that due to Lemma \ref{lem:heightcomp} and the properties of heights we obtain
\begin{align*}
\max\{D h(\eta_3), |\log\eta_3|, 0.16\}&\leq 2(m_1-m_2+1)\log b+2(n_1-n_2+1)\log \alpha\\
&\leq 1.1\cdot 10^{10} \log b\log \alpha \log(27.62 n_1) =:A_3.
\end{align*}
Thus we obtain $E\leq 2n_1$ as before, and from Matveev's theorem
\begin{equation*}
	3.2\cdot 10^{21} [\log b\log \alpha \log(27.62 n_1)]^2 \geq n_1 \log \gamma - \log 8 \geq n_1 \frac{\epsilon}{2} \log \alpha -\log 8,
\end{equation*}
which yields
\begin{equation}\label{eq:n1-bound}
6.42\cdot 10^{21} \frac{\log \alpha}{\epsilon} [\log b \log(27.62 n_1)]^2>n_1.
\end{equation}

If we put $n'=27.62 n_1$ and apply Lemma \ref{lem:log-sol} to \eqref{eq:n1-bound}, then we end up with
\begin{equation*}
	n'< 7.12\cdot 10^{23} \frac{\log \alpha}{\epsilon} (\log b)^2
	\log\left(7.12\cdot 10^{23} \frac{\log \alpha}{\epsilon}(\log b)^2\right)^2
\end{equation*}
which yields the content of Proposition \ref{prop:n1bound-logalpha}.

\section{Combining linear forms of logarithms}\label{sec:three-sol}

As done before, let us assume that Diophantine equation \eqref{eq:PillaiVn} has three solutions $(n_1,m_1)$, $(n_2,m_2)$, $(n_3,m_3)$ with $n_1>n_2>n_3\geq N_0$.
Again we assume that Assumption A1 or A2 holds.

Let us reconsider Inequality \eqref{eq:L1-exp} with $n_1,m_1,n_2,m_2$ replaced by $n_2,m_2,n_3,m_3$, respectively. Then we obtain 
\begin{equation}\label{eq:exp-L2}
\begin{split}
\left|\frac{b^{m_2}}{\alpha^{n_2}}-1\right|&\leq \alpha^{n_3-n_2}+\frac{b^{m_3}}{\alpha^{n_2}}+2\gamma^{-n_2}\\
&\leq 3 \max\left\{\alpha^{n_3-n_2},\frac{5}{2}b^{m_3-m_2},2\gamma^{-n_2}\right\}.
\end{split}
\end{equation}
Note that with the assumption that three solutions exist we can apply Lemma~\ref{lem:n1-m1-relation} to $\frac{b^{m_2}}{\alpha^{n_2}}$ but not to
$\frac{b^{m_3}}{\alpha^{n_3}}$.

Let us assume for the next paragraphs that
\begin{equation}\label{eq:max}
 M_0:=\max\left\{3\alpha^{n_3-n_2},\frac{15}{2}b^{m_3-m_2},6\gamma^{-n_2}, 7\alpha^{n_2-n_1},4\gamma^{-n_1}\right\}<\frac{1}{2}.
\end{equation}
Then, by applying Lemma \ref{lem:log-est} to \eqref{eq:L1-exp} and \eqref{eq:exp-L2}, we obtain the system of inequalities
\begin{align*}
 |\underbrace{m_1\log b-n_1\log \alpha}_{=:\Lambda_1}|&\leq \max\left\{14 \alpha^{n_2-n_1},8\gamma^{-n_1}\right\},\\
 |\underbrace{m_2\log b-n_2\log \alpha}_{=:\Lambda_2}|&\leq\max\left\{6\alpha^{n_3-n_2},15b^{m_3-m_2},12\gamma^{-n_2}\right\}.
\end{align*}
We eliminate the term $\log \alpha$ from these inequalities by considering $\Lambda_0=n_2\Lambda_1-n_1\Lambda_2$ and obtain the inequality
\begin{equation}\label{eq:LinForm-L0}
\begin{split}
 |\Lambda_0|&=|(n_2m_1-n_1m_2)\log b|\\
 &\leq \max\left\{12n_1 \alpha^{n_3-n_2},30 n_1 b^{m_3-m_2},24 n_1 \gamma^{-n_2}, 28 n_2\alpha^{n_2-n_1},16 n_2\gamma^{-n_1} \right\}.
 \end{split}
\end{equation}

Let us write $M$ for the maximum on the right hand side of \eqref{eq:LinForm-L0}. If $n_2m_1-n_1m_2\neq 0$, we obtain the inequality $\log b \leq M$. Since we will study the case that $n_2m_1-n_1m_2=0$ in Section \ref{sec:degenerate-case},
we will assume for the rest of this section that $n_2m_1-n_1m_2\neq 0$, i.e.\ we have $\log b \leq M$. Therefore we have to consider five different cases. In each case we want to find an upper bound for $\log \alpha$ if possible:

\begin{description}
 \item [The case $M=12n_1 \alpha^{n_3-n_2}$] In this case we get
 \begin{equation*}
 	\log \log b \leq \log (12n_1)-(n_2-n_3)\log \alpha
 \end{equation*}
 which yields
 \begin{equation*}
 	\log \alpha \leq \log \left(12n_1/\log b\right)<\log (17.4 n_1),
 \end{equation*}
 since we assume $b\geq 2$.
 \item [The case $M=30 n_1 b^{m_3-m_2}$] In this case we obtain
  \begin{equation*}
  	\log \log b \leq \log (30n_1)-(m_2-m_3)\log b
  \end{equation*}
  which yields
  \begin{equation*}
  	m_2-m_3 \leq \frac{\log (30n_1/\log b)}{\log b}<1.45 \log(43.3 n_1).
  \end{equation*}
  To obtain from this inequality a bound for $\log \alpha$ is not straight forward and we will deal with this case in Section \ref{sec:case-m2-m3}.
 \item [The case $M=24 n_1 \gamma^{-n_2}$] This case implies
 \begin{equation*}
 	\log \log b \leq \log (24 n_1)-n_2\log \gamma \leq \log (24 n_1)-n_2 \frac{\epsilon}{2} \log \alpha
 \end{equation*}
 and we obtain
 \begin{equation*}
 	\log \alpha \leq \frac{2\log \left(24n_1/\log b\right)}{\epsilon}<\frac{2\log (34.7 n_1)}{\epsilon}.
 \end{equation*}
 \item [The case $M=28 n_2\alpha^{n_2-n_1}$] By a similar computation as in the first case, we obtain in this case the inequality
 \begin{equation*}
 	\log \alpha <\log (40.4 n_2)< \log (40.4 n_1).
 \end{equation*}
 \item [The case $M=16 n_2\gamma^{-n_1}$] Almost the same computations as in the case that $M=24 n_1 \gamma^{-n_2}$ lead to
 \begin{equation*}
 	\log \alpha <\frac{2\log (23.1 n_1)}{\epsilon}.
 \end{equation*}
 \end{description}
 
 In the case that \eqref{eq:max} does not hold, i.e.\ that $M_0\geq 1/2$, we obtain by similar computations in each of the five possibilities the following inequalities:
 
 \begin{description}
 \item [The case $M_0=3 \alpha^{n_3-n_2}$] $\log \alpha \leq \log 6$;
 \item [The case $M_0=\frac{15}{2} b^{m_3-m_2}$] $m_2-m_3\leq 3$;
 \item [The case $M_0=6 \gamma^{-n_2}$] $\log\alpha \leq \frac{\log 144}{\epsilon}$;
 \item [The case $M_0=7 \alpha^{n_2-n_1}$] $\log \alpha \leq \log 14$;
 \item [The case $M_0=4 \gamma^{-n_1}$] $\log\alpha \leq \frac{\log 64}{\epsilon}$.
 \end{description}

 Let us recap what we have proven so far in the following lemma:
 
 \begin{lemma}\label{lem:three-cases}
  Assume that Assumption A1 or A2 holds and assume that Diophantine equation \eqref{eq:PillaiVn} has three solutions $(n_1,m_1)$, $(n_2,m_2)$, $(n_3,m_3)$ with $n_1>n_2>n_3\geq N_0$. Then one of the following three possibilities holds:
  \begin{enumerate}[label={\roman*)},leftmargin=*,widest=999]
   \item $n_2m_1-n_1m_2=0$;
   \item $m_2-m_3 < 1.45 \log(43.3 n_1)$;
   \item $\log \alpha <\frac{2\log (34.7 n_1)}{\epsilon}.$
  \end{enumerate}
 \end{lemma}
 
 Since we will deal with the first and second possiblity in the next sections, we close this section by proving that the last possibility implies Theorems \ref{th:n1-bound} and \ref{th:absolute}. Therefore let us plug in the upper bound for $\log \alpha$ into Inequality \eqref{eq:n1-bound} to obtain
 \begin{equation*}
 	1.29\cdot 10^{22} \left(\frac{\log b}{\epsilon}\right)^2 \log(34.7 n_1)^3>n_1.
 \end{equation*}
 Writing $n'=34.7 n_1$, this inequality turns into
 \begin{equation*}
 	4.48\cdot 10^{23} \left(\frac{\log b}{\epsilon}\right)^2 \log(n')^3>n'
 \end{equation*}
 and an application of Lemma \ref{lem:log-sol} implies
 \begin{equation*}
 	n_1 < 1.04 \cdot 10^{23} \left(\frac{\log b}{\epsilon}\right)^2
 	\log\left( 1.21 \cdot 10^{25} \left(\frac{\log b}{\epsilon}\right)^2\right)^3.
 \end{equation*}
Thus Theorem \ref{th:n1-bound} is proven in this case.

Now let us assume that Assumption A2 holds. By Remark \ref{rem:A1-A2} we get the bound
\begin{equation*}
	n_1<4.16 \cdot 10^{23} \left(\log b\right)^2
	\log\left( 4.84 \cdot 10^{25} \left(\log b\right)^2\right)^3.
\end{equation*}
If we insert our upper bound for $n_1$ into the upper bound for $\log \alpha$, we obtain
\begin{align*}
	\log A &< \log \alpha + \log 2 \leq 5 \log (34.7 n_1) \\
	&\leq 5 \log \left( 1.45 \cdot 10^{25} \left(\log b\right)^2
	\log\left( 4.84 \cdot 10^{25} \left(\log b\right)^2\right)^3 \right).
\end{align*}
This proves Theorem \ref{th:absolute} in the current case.

\section{The case $n_1m_2-n_2m_1=0$}\label{sec:degenerate-case}

We distinguish between the cases $c\geq 0$ and $c<0$.

\subsection{The case $c\geq 0$}

In this case we have
\begin{equation}\label{eq:c-alpha-n3-bound}
	0\leq c=V_{n_3}-b^{m_3}<\alpha^{n_3}+\beta^{n_3}<2\alpha^{n_3}.
\end{equation}
Furthermore it holds
\begin{equation*}
	\frac{b^{m_1}}{\alpha^{n_1}}=1+\frac{\beta^{n_1}-c}{\alpha^{n_1}}\qquad \text{ as well as } \qquad\frac{b^{m_2}}{\alpha^{n_2}}=1+\frac{\beta^{n_2}-c}{\alpha^{n_2}}.
\end{equation*}
Since 
\begin{equation*}
	\left|\frac{\beta^{n_2}-c}{\alpha^{n_2}}\right|\leq \left(\frac{|\beta|}{\alpha}\right)^{n_2}+2\alpha^{n_3-n_2}<\frac{1}{2}
\end{equation*}
holds under our assumptions, we may apply Lemma \ref{lem:log-est} to get the two inequalities
\begin{align*}
 \frac{2}{9} \left(\frac{\beta^{n_1}-c}{\alpha^{n_1}} \right)^2& \leq \frac{\beta^{n_1}-c}{\alpha^{n_1}} -\left(m_1\log b-n_1\log \alpha\right) \leq 2 \left(\frac{\beta^{n_1}-c}{\alpha^{n_1}} \right)^2,\\
 \frac{2}{9} \left(\frac{\beta^{n_2}-c}{\alpha^{n_2}} \right)^2& \leq \frac{\beta^{n_2}-c}{\alpha^{n_2}}-\left(m_2\log b-n_2\log \alpha\right) \leq 2 \left(\frac{\beta^{n_2}-c}{\alpha^{n_2}} \right)^2.
\end{align*}
Multiplying the first inequality by $n_2$ and the second one by $n_1$ as well as forming the difference afterwards yields
\begin{equation}\label{eq:log-est-for-alpha}
\begin{split}
 \frac{2n_2}{9} \left(\frac{\beta^{n_1}-c}{\alpha^{n_1}} \right)^2-2n_1\left(\frac{\beta^{n_2}-c}{\alpha^{n_2}} \right)^2
 &\leq n_2\frac{\beta^{n_1}-c}{\alpha^{n_1}}-n_1\frac{\beta^{n_2}-c}{\alpha^{n_2}} \\
 &\leq 2n_2\left(\frac{\beta^{n_1}-c}{\alpha^{n_1}} \right)^2- \frac{2n_1}{9} \left(\frac{\beta^{n_2}-c}{\alpha^{n_2}} \right)^2.
 \end{split}
\end{equation}
Let us note that $(a+b)^2\leq 4\max\{|a|^2,|b|^2\}$, and therefore we obtain
\begin{align*}
 c\left(\frac{n_1}{\alpha^{n_2}}-\frac{n_2}{\alpha^{n_1}}\right)&\leq \frac{n_2|\beta|^{n_1}}{\alpha^{n_1}}+\frac{n_1|\beta|^{n_2}}{\alpha^{n_2}}+8n_2\max\left\{\frac{|\beta|^{2n_1}}{\alpha^{2n_1}},\frac{c^2}{\alpha^{2n_1}}\right\}\\
 &\leq 10n_1\max\left\{\frac{|\beta|^{n_2}}{\alpha^{n_2}},\frac{c^2}{\alpha^{2n_2}} \right\}.
\end{align*}
Together with the estimate
\begin{equation*}
	\frac{n_1}{\alpha^{n_2}}-\frac{n_2}{\alpha^{n_1}}>\frac{n_1}{\alpha^{n_2}}\left(1-\frac{1}{\alpha}\right)>\frac 78 \cdot\frac{n_1}{\alpha^{n_2}}
\end{equation*}
this implies
\begin{equation}\label{eq:c-est-for-c-pos}
 c< 12\max\left\{|\beta|^{n_2},\frac{c^2}{\alpha^{n_2}}\right\}.
\end{equation}
Let us assume for the moment that the maximum is $\frac{c^2}{\alpha^{n_2}}$. Then we obtain
\begin{equation*}
	\alpha^{n_2}< 12 c < 24\alpha^{n_3}
\end{equation*}
which implies $\alpha<24$ and thus Theorem \ref{th:absolute}. Plugging in $\alpha<24$ in Proposition~\ref{prop:n1bound-logalpha} yields the content of Theorem \ref{th:n1-bound} in this case.

Therefore we assume now $c<12|\beta|^{n_2}$. By Proposition \ref{prop:c-low-bound} we obtain
\begin{equation*}
	\alpha^{n_1-K_0\log(27.62 n_1)}-|\beta|^{n_1}<c<12|\beta|^{n_2}
\end{equation*}
which yields
\begin{equation*}
	\left(n_1-K_0\log(27.62 n_1)\right) \log \alpha <\log 13 +n_1\max\{\log|\beta|,0\}.
\end{equation*}
Note that, due to our assumptions, we have the bound
\begin{equation*}
	\alpha^{1-\frac{\epsilon}{4}} \geq 8\alpha^{1-\epsilon} > 4A^{1-\epsilon} > \frac{2\alpha}{A^{\epsilon}}>|\beta|
\end{equation*}
which implies $(1-\frac{\epsilon}{4}) \log \alpha > \log |\beta|$.
Thus we get
\begin{equation*}
	n_1 \frac{\epsilon}{4} \log \alpha < K_0 \log \alpha \log(27.62 n_1)+\log 13
\end{equation*}
and
\begin{equation*}
	n_1< 1.08 \cdot 10^{10} \frac{\log b}{\epsilon} \log(27.62 n_1).
\end{equation*}
As previously, solving this inequality with the help of Lemma \ref{lem:log-sol} yields
\begin{equation}\label{eq:n1-bound-deg-case}
 n_1< 2.17 \cdot 10^{10} \frac{\log b}{\epsilon}\log\left(2.99 \cdot 10^{11} \frac{\log b}{\epsilon} \right)
\end{equation}
which proves Theorem \ref{th:n1-bound} in this case.

So we may now assume that Assumption A2 holds and the bound for $n_1$ with $\epsilon=\frac{1}{2}$ is valid. This yields
\begin{equation}\label{eq:n1-bound-deg-case-A2}
n_1< 4.34 \cdot 10^{10}\log b\log\left(5.98 \cdot 10^{11} \log b\right).
\end{equation}
Note that under Assumption A2 we have $|\beta|<2\kappa$, by Lemma \ref{lem:bound-alpha-beta-case2}. Hence the quantities $c,|\beta|^{n_1},|\beta|^{n_2}$ are bounded by absolute, effectively computable constants.

Let us consider the case $|c-\beta^{n_2}|\geq 2(c+|\beta|^{n_1})\alpha^{n_2-n_1}$.
Note that $ \beta^{n_2} \neq c $ by the usual Galois conjugation argument.
If $ c > \beta^{n_2} $, then \eqref{eq:log-est-for-alpha} gives us
\begin{align*}
 (2n_1-n_2)\frac{c+|\beta|^{n_1}}{\alpha^{n_1}} &\leq n_2\frac{\beta^{n_1}-c}{\alpha^{n_1}}-n_1\frac{\beta^{n_2}-c}{\alpha^{n_2}}
 \\
 &\leq 2n_2\left(\frac{\beta^{n_1}-c}{\alpha^{n_1}} \right)^2- \frac{2n_1}{9} \left(\frac{\beta^{n_2}-c}{\alpha^{n_2}} \right)^2
 \leq 2n_1\left(\frac{|\beta|^{n_1}+c}{\alpha^{n_1}} \right)^2
\end{align*}
which yields
\begin{equation*}
	\alpha^{n_1}\leq 2(|\beta|^{n_1}+c).
\end{equation*}
As $c<2\alpha^{n_3}$ (see Inequality \eqref{eq:c-alpha-n3-bound}) we obtain
\begin{equation*}
	0.2 \alpha^{n_1} < \alpha^{n_1}-2c \leq 2|\beta|^{n_1} < 2(2\kappa)^{n_1}
\end{equation*}
which yields $\alpha< 20 \kappa$.

If $ c < \beta^{n_2} $, then \eqref{eq:log-est-for-alpha} gives us
\begin{align*}
	\left(n_1-\frac{n_2}{2}\right)\frac{\beta^{n_2}-c}{\alpha^{n_2}} &\leq n_1\frac{\beta^{n_2}-c}{\alpha^{n_2}} - n_2\frac{\beta^{n_1}-c}{\alpha^{n_1}}
	\\
	&\leq 2n_1\left(\frac{\beta^{n_2}-c}{\alpha^{n_2}} \right)^2- \frac{2n_2}{9} \left(\frac{\beta^{n_1}-c}{\alpha^{n_1}} \right)^2
	\leq 2n_1\left(\frac{\beta^{n_2}-c}{\alpha^{n_2}} \right)^2
\end{align*}
which implies
\begin{equation*}
	1 \leq 4\cdot \frac{\beta^{n_2}-c}{\alpha^{n_2}} \leq 4\cdot \frac{|\beta|^{n_2}}{\alpha^{n_2}} \leq 4\cdot \frac{|\beta|}{\alpha} < 8 \kappa \alpha^{-1}
\end{equation*}
and hence $ \alpha < 8\kappa $.

Thus we may now assume $|c-\beta^{n_2}|< 2(c+|\beta|^{n_1})\alpha^{n_2-n_1}$. Let us note that under the assumption $\alpha>2(c+\max\{1,|\beta|\}^{n_1})$ we can deduce
\begin{equation*}
	\left|\frac{b^{m_3}}{\alpha^{n_3}}-1\right| \leq \frac{c+|\beta|^{n_3}}{\alpha^{n_3}}<\frac{1}{2},
\end{equation*}
and otherwise we would get the constant $ C_2 $ in Theorem \ref{th:absolute} (cf.\ the calculations below).
Then, by Lemma \ref{lem:log-est}, we get
\begin{equation*}
	|\underbrace{m_3\log b-n_3\log \alpha}_{=:\Lambda_3}| \leq \frac{2c+2|\beta|^{n_3}}{\alpha^{n_3}} \leq \frac{24|\beta|^{n_2}+2|\beta|^{n_3}}{\alpha^{n_3}} \leq \frac{26(2\kappa)^{n_1}}{\alpha^{n_3}}.
\end{equation*}
Recalling from the beginning of Section \ref{sec:three-sol} the bound
\begin{equation*}
	|\underbrace{m_1\log b-n_1\log \alpha}_{=:\Lambda_1}| \leq \max\left\{14 \alpha^{n_2-n_1},8\gamma^{-n_1}\right\},
\end{equation*}
we can again eliminate the term $\log \alpha$ from these inequalities by considering the form $\Lambda_0'=n_3\Lambda_1-n_1\Lambda_3$ and obtain the inequality
\begin{equation}\label{eq:LinForm-Lp0}
	\begin{split}
		|\Lambda_0'|&=|(n_3m_1-n_1m_3)\log b|\\
		&\leq \max\left\{52n_1(2\kappa)^{n_1} \alpha^{-n_3}, 28 n_3\alpha^{n_2-n_1},16 n_3\gamma^{-n_1} \right\} \\
		&\leq \max\left\{52n_1(2\kappa)^{n_1} \alpha^{-n_3}, 28 n_3\alpha^{n_2-n_1},16 n_3\alpha^{-n_1/4} \right\}.
	\end{split}
\end{equation}
If $ n_3m_1-n_1m_3 \neq 0 $, then we have
\begin{equation*}
	\log b \leq \max\left\{52n_1(2\kappa)^{n_1} \alpha^{-n_3}, 28 n_3\alpha^{n_2-n_1},16 n_3\alpha^{-n_1/4} \right\}
\end{equation*}
which yields $ \log \alpha \leq 5 + \log n_1 + n_1 \log(4\kappa) $, and together with the bound \eqref{eq:n1-bound-deg-case-A2} this gives us constant $ C_2 $ in Theorem \ref{th:absolute}.

Hence we can assume $ n_3m_1-n_1m_3 = 0 $ and replace in the discussion above $m_2$ by $m_3$ as well as $n_2$ by $n_3$ in order to get $|c-\beta^{n_3}|< 2(c+|\beta|^{n_1})\alpha^{n_3-n_1}$.
Thus altogether we obtain
\begin{equation}\label{eq:beta-bound}
|\beta|^{n_3}\left|\beta^{n_2-n_3}-1\right|=\left|\beta^{n_2}-\beta^{n_3}\right|<4(c+|\beta|^{n_1})\alpha^{n_2-n_1}.
\end{equation}

From \eqref{eq:beta-bound} we deduce that one of the two factors of the left hand side is smaller than $2\alpha^{\frac{n_2-n_1}2}\sqrt{c+|\beta|^{n_1}}$.
By thinking of constant $ C_2 $ in Theorem \ref{th:absolute}, we may assume $\alpha >64(c +|\beta|^{n_1})$. So we have $2\alpha^{\frac{n_2-n_1}2}\sqrt{c+|\beta|^{n_1}} < \frac{1}{4}$.
Let us first assume that
\begin{equation*}
	|\beta|^{n_3} < 2\alpha^{\frac{n_2-n_1}2}\sqrt{c+|\beta|^{n_1}}.
\end{equation*}
This implies $|\beta|^{n_3}< \frac 14$ and further
\begin{equation*}
	c< |\beta|^{n_3}+2(c+|\beta|^{n_1})\alpha^{n_3-n_1}<\frac 14 +\frac 18<1.
\end{equation*}
Therefore we have $c=0$. But Lemma \ref{lem:c0} states that there cannot be three solutions for $ c=0 $.

Now we may assume
\begin{equation*}
	\left|\beta^{n_2-n_3}-1\right| < 2\alpha^{\frac{n_2-n_1}2}\sqrt{c+|\beta|^{n_1}}.
\end{equation*}
Note that, in particular, we have $ \beta^{n_2-n_3} > 0 $, and the above inequality implies
\begin{equation*}
	1-2\alpha^{\frac{n_2-n_1}2}\sqrt{c+|\beta|^{n_1}} < |\beta|^{n_2-n_3} < 1+2\alpha^{\frac{n_2-n_1}2}\sqrt{c+|\beta|^{n_1}}.
\end{equation*}
Using the binomial series expansion for $ (1\pm x)^r $ with exponent $r=\frac{1}{n_2-n_3}$ yields
\begin{equation*}
	||\beta|-1|<2\alpha^{\frac{n_2-n_1}2}\sqrt{c+|\beta|^{n_1}}.
\end{equation*}
Assuming $\alpha>64n_2^2(c+|\beta|^{n_1})$, we obtain by an application of Lemma \ref{lem:exp-est} that
\begin{equation*}
	||\beta|^{n_2}-1| \leq 5.2 n_2 \alpha^{\frac{n_2-n_1}2}\sqrt{c+|\beta|^{n_1}}.
\end{equation*}
This together with our assumption $ |c-|\beta|^{n_2}| \leq |c-\beta^{n_2}|< 2(c+|\beta|^{n_1})\alpha^{n_2-n_1}$ gives us
\begin{equation*}
	|c-1|<5.2 n_2 \alpha^{\frac{n_2-n_1}2}\sqrt{c+|\beta|^{n_1}}+ 2(c+|\beta|^{n_1})\alpha^{n_2-n_1}<\frac 12
\end{equation*}
provided that
\begin{equation*}
	\alpha \geq 433 n_2^2(c+|\beta|^{n_1}).
\end{equation*}

Thus we may assume $c=1$ provided that $\alpha$ is large enough. But this also implies
\begin{equation*}
	|1-\beta^{n_3}|<2(1+|\beta|^{n_1})\alpha^{n_3-n_1}\leq 2(1+|\beta|^{n_1})\alpha^{-2}.
\end{equation*}
If $\beta^{n_3}<0$, we get
\begin{equation*}
	\alpha<\sqrt{2(1+|\beta|^{n_1})} \leq 2 (c+|\beta|^{n_1}).
\end{equation*}
Therefore we may assume $\beta^{n_3}=|\beta|^{n_3}$ is positive. Since for any real numbers $x>0$ and $n\geq 1$ we have $|1-x| \leq |1-x^n|$, we obtain from Lemma \ref{lem:beta-1-bound} together with Lemma \ref{lem:alpha-A-est}
\begin{equation*}
	\frac{2}{4\alpha+5}<\frac{2}{2A+5}\leq |1-|\beta||\leq |1-\beta^{n_3}|<2(c+|\beta|^{n_1})\alpha^{-2}.
\end{equation*}
Hence we get
\begin{equation*}
	\alpha<9(c+|\beta|^{n_1})
\end{equation*}
in this case.

Let us summarize what we have proven so far: In the case $c\geq 0$ under Assumption A2 we cannot have three solutions to \eqref{eq:PillaiVn} for
\begin{equation*}
	\alpha\geq 433 n_2^2(c+|\beta|^{n_1}).
\end{equation*}
So if there exist three solutions, then we have
\begin{equation*}
	\alpha < 433 n_2^2(c+|\beta|^{n_1}) < 5629 n_1^2(2\kappa)^{n_1}.
\end{equation*}
With \eqref{eq:n1-bound-deg-case-A2} this implies
\begin{align*}
\log A&<n_1\log (4\kappa)+2\log n_1+\log 11258\\
&< 4.35 \cdot 10^{10}\log (4\kappa) \log b\log\left(5.98 \cdot 10^{11} \log b\right) 
\end{align*}
and Theorem \ref{th:absolute} is proven in that case.

\subsection{The case $c< 0$}

The case $c<0$ can be treated with similar arguments. Therefore we will only point out the differences.

We start with the observation
\begin{equation*}
	0<|c|=-c=b^{m_3}-V_{n_3}<b^{m_3}
\end{equation*}
and write again
\begin{equation*}
	\frac{b^{m_1}}{\alpha^{n_1}}=1+\frac{\beta^{n_1}-c}{\alpha^{n_1}}\qquad \text{ as well as } \qquad\frac{b^{m_2}}{\alpha^{n_2}}=1+\frac{\beta^{n_2}-c}{\alpha^{n_2}}.
\end{equation*}
Note that, using Lemma \ref{lem:n1-m1-relation},
\begin{equation*}
	\left|\frac{\beta^{n_2}-c}{\alpha^{n_2}}\right|
	\leq \left(\frac{|\beta|}{\alpha}\right)^{n_2}+\frac{|c|}{\alpha^{n_2}}
	< \frac{1}{16} + \frac{b^{m_3}}{\alpha^{n_2}}
	< \frac{1}{16} + \frac{5}{2} b^{m_3-m_2} < \frac{1}{2}
\end{equation*}
holds under our assumptions if in addition $ m_2-m_3 \geq 2 $. The case $ m_2-m_3 =1 $ is included in the next section.
Thus we get again the inequality chain \eqref{eq:log-est-for-alpha} and furthermore the bound
\begin{equation*}
	|c|< 12\max\left\{|\beta|^{n_2},\frac{|c|^2}{\alpha^{n_2}}\right\}.
\end{equation*}

If the maximum is $ \frac{|c|^2}{\alpha^{n_2}} $, then we have
\begin{equation*}
	\frac{2}{5} b^{m_2} < \alpha^{n_2} < 12|c| < 12 b^{m_3}
\end{equation*}
which implies $ m_2-m_3 \leq 3 $. This will be handled in Section \ref{sec:case-m2-m3}.
Therefore we may now again assume $ |c| < 12|\beta|^{n_2} $.
In the same way as in the case $ c \geq 0 $ we obtain again the upper bound \eqref{eq:n1-bound-deg-case} proving Theorem \ref{th:n1-bound} also in the case $c<0$. Moreover, we get under Assumption A2 the same upper bound \eqref{eq:n1-bound-deg-case-A2} for $n_1$. In particular, the quantities $|c|,|\beta|^{n_1},|\beta|^{n_2}$ are bounded by absolute, effectively computable constants.

Apart from the special case
\begin{equation*}
	\alpha^{n_1} \leq 2(|\beta|^{n_1}+|c|)
\end{equation*}
which, by Lemma \ref{lem:n1-m1-relation}, yields
\begin{equation*}
	\frac{1}{15}\alpha^{n_1} < \frac{8}{45} b^{m_1} < \alpha^{n_1}-2b^{m_3} < \alpha^{n_1}-2|c| \leq 2|\beta|^{n_1} < 2(2\kappa)^{n_1}
\end{equation*}
and thus $ \alpha < 60 \kappa $,
as well as the consideration of $ c=-1 $ instead of $ c=1 $,
we essentially only have to replace some (not all!) occurrences of $ c $ by $ |c| $ in order to perform the analogous arguments as we used in the case $ c \geq 0 $.
Hence Theorem~\ref{th:absolute} is proven in this case as well.

Let us summarize what we have proven so far:

\begin{lemma}\label{lem:last-cases}
  Assume that Assumption A1 or A2 holds and assume that Diophantine equation \eqref{eq:PillaiVn} has three solutions $(n_1,m_1)$, $(n_2,m_2)$, $(n_3,m_3)$ with $n_1>n_2>n_3\geq N_0$. Then at least one of the following three possibilities holds:
  \begin{enumerate}[label={\roman*)},leftmargin=*,widest=999]
   \item Assumption A1 holds and
    \begin{equation*}
    	n_1 < 1.04 \cdot 10^{23} \left(\frac{\log b}{\epsilon}\right)^2
    	\log\left( 1.21 \cdot 10^{25} \left(\frac{\log b}{\epsilon}\right)^2\right)^3;
    \end{equation*}
   \item Assumption A2 holds and
    \begin{equation*}
    	\log A < 4.35 \cdot 10^{10}\log (4\kappa) \log b\log\left(5.98 \cdot 10^{11} \log b\right)
    \end{equation*}
	or
	\begin{equation*}
		\log A < 5 \log \left( 1.45 \cdot 10^{25} \left(\log b\right)^2
		\log\left( 4.84 \cdot 10^{25} \left(\log b\right)^2\right)^3 \right);
	\end{equation*}
   \item $m_2-m_3<1.45 \log(43.3 n_1)$.
  \end{enumerate}
 \end{lemma}

\section{The case $m_2-m_3\ll \log n_1$}\label{sec:case-m2-m3}

In view of Theorems \ref{th:n1-bound} and \ref{th:absolute} and Lemma \ref{lem:last-cases} we may assume that Assumption A1 or A2 holds and that $m_2-m_3<1.45 \log(43.3 n_1)$.

First we reconsider inequality \eqref{eq:L1-exp} and note that $7\max\left\{\alpha^{n_2-n_1},\gamma^{-n_1}\right\}\geq \frac{1}{2}$ implies either $\alpha \leq 14$ or $A^{\epsilon} \leq 28$. In the first cases we have an upper bound for $\alpha$ and by Proposition \ref{prop:n1bound-logalpha} also an upper bound for $n_1$; the second case contradicts Assumption A1 and A2 respectively. Thus Theorems \ref{th:n1-bound} and \ref{th:absolute} are shown in those situations. Now we may apply Lemma \ref{lem:log-est} and obtain, as in Section \ref{sec:firstbound}, the inequality
\begin{equation}\label{eq:Lambda1}
|\underbrace{n_1\log\alpha-m_1\log b}_{=:\Lambda_1}| \leq 14 \max\left\{\alpha^{n_2-n_1},\gamma^{-n_1}\right\}.
\end{equation}

Next, let us consider the inequality
\begin{align*}
\left|\frac{\alpha^{n_2}}{b^{m_2}-b^{m_3}}-1\right|=\left|\frac{\alpha^{n_2}}{b^{m_3}(b^{m_2-m_3}-1)}-1\right|&\leq \frac{\alpha^{n_3}+|\beta|^{n_2}+|\beta|^{n_3}}{b^{m_2}-b^{m_3}}\\
&< 6\alpha^{n_3-n_2}+12\gamma^{-n_2}\\
&\leq 18\max\{\alpha^{n_3-n_2},\gamma^{-n_2}\}.
\end{align*}
In particular, note that $b^{m_2}-b^{m_3}\geq \frac{1}{2} b^{m_2} > \frac{3}{16} \alpha^{n_2}$ by Lemma \ref{lem:n1-m1-relation}.
Assuming that $18\alpha^{n_3-n_2} \geq \frac{1}{2}$ yields $\alpha \leq 36$ which implies by Proposition \ref{prop:n1bound-logalpha} Theorems \ref{th:n1-bound} and \ref{th:absolute}. Similarly, using $ n_2\geq 2 $, the assumption $18\gamma^{-n_2} \geq \frac{1}{2}$ gives us either $ \alpha \leq 6 $ or $A^{\epsilon} \leq 12$ and we are done as well. Thus we may apply Lemma \ref{lem:log-est} and obtain
\begin{equation}\label{eq:Lambda2}
|\underbrace{n_2\log\alpha-m_3\log b-\log(b^{m_2-m_3}-1)}_{=:\Lambda_2}| \leq 36\max\{\alpha^{n_3-n_2},\gamma^{-n_2}\}.
\end{equation}

Eliminating the term $\log \alpha$ from the linear forms $\Lambda_1$ and $\Lambda_2$ by considering $\Lambda=n_1\Lambda_2-n_2\Lambda_1$ yields together with \eqref{eq:Lambda1} and \eqref{eq:Lambda2} the bound
\begin{equation}\label{eq:Lambda}
\begin{split}
|\Lambda|&\leq 36n_1\max\{\alpha^{n_3-n_2},\gamma^{-n_2}\}+14n_2\max\{\alpha^{n_2-n_1},\gamma^{-n_1}\}\\
&\leq 50n_1\max\{\alpha^{n_3-n_2},\gamma^{-n_2},\alpha^{n_2-n_1}\}\\
&\leq 200n_1\alpha^{-\epsilon},
\end{split}
\end{equation}
where
\begin{equation*}
	\Lambda=(m_1n_2-m_3n_1)\log b-n_1\log(b^{m_2-m_3}-1).
\end{equation*}
Now we have to distinguish between the cases $\Lambda=0$ and $\Lambda\neq0$.

\subsection{The case $\Lambda= 0$}

Since $n_1\neq 0$ this case can only occur if $b$ and $b^{m_2-m_3}-1$ are multiplicatively dependent. This is only possible if $b=2$ and $m_2-m_3=1$, i.e.\ if $b^{m_2-m_3}-1=1$. Note that our assumptions imply $c\geq 0$ if $b=2$. Therefore we obtain
\begin{equation*}
	0\leq c =V_{n_3}-b^{m_3}=\alpha^{n_3}+\beta^{n_3}-b^{m_3}
\end{equation*}
which implies $b^{m_3}\leq 2\alpha^{n_3}$.

From Lemma \ref{lem:n1-m1-relation} we know that $b^{m_2}>\frac{3}{8} \alpha^{n_2}$. Hence, using the facts $b=2$ and $m_2=m_3+1$, we get the inequality
\begin{equation*}
	\frac{3}{8} \alpha^{n_2}<b^{m_2}=2b^{m_3}\leq 4\alpha^{n_3}
\end{equation*}
which implies $\alpha^{n_2-n_3}<11$ and thus $\alpha<11$. An application of Proposition~\ref{prop:n1bound-logalpha} yields Theorems \ref{th:n1-bound} and \ref{th:absolute}.

\begin{remark}
 Let us note that in the case $c<0$ the argument above does not work. This is the reason why we exclude $b=2$ if $c<0$.
\end{remark}

\subsection{The case $\Lambda\neq 0$}

Here we may apply Matveev's theorem, Lemma \ref{lem:matveev}, to $\Lambda$. Note that the case $b^{m_2-m_3}-1=1$ can be excluded by the previous subsection.

First, let us find an upper bound for $|m_1n_2-m_3n_1|$. We deduce from \eqref{eq:Lambda} the bound
\begin{align*}
	|m_1n_2-m_3n_1|\log b &\leq n_1\log(b^{m_2-m_3}-1) + 200n_1 \\
	&\leq n_1(m_2-m_3)\log b + 200n_1
\end{align*}
which implies
\begin{equation*}
	|m_1n_2-m_3n_1|\leq 90 n_1 \log(43.3 n_1).
\end{equation*}
Furthermore, using Lemma \ref{lem:heightcomp}, we have
\begin{align*}
	\max\{Dh(b^{m_2-m_3}- 1),|\log (b^{m_2-m_3}- 1)|,0.16\} &\leq 2(m_2-m_3+1)\log b \\
	&\leq 3.5 \log(43.3 n_1)\log b.
\end{align*}
Therefore we may choose $E=52 n_1$ in Lemma \ref{lem:matveev}. 

Now we obtain by Matveev's theorem
\begin{equation*}
	4.69\cdot 10^{9} \log(719 n_1)\log(43.3 n_1) (\log b)^2 \geq -\log|\Lambda|
\end{equation*}
and then
\begin{equation*}
	4.69\cdot 10^{9} [\log(719 n_1) \log b]^2 \geq -\log|\Lambda|.
\end{equation*}
Together with the upper bound for $|\Lambda|$ this yields
\begin{equation*}
	4.69\cdot 10^{9} [\log(719 n_1) \log b]^2 \geq \epsilon \log \alpha - \log (200n_1)
\end{equation*}
and thus
\begin{equation}\label{eq:alpha-bound-last-case}
\log \alpha < 4.7 \cdot 10^{9} \frac{1}{\epsilon}[\log(719 n_1) \log b]^2.
\end{equation}

Similar as in Section \ref{sec:three-sol} we plug in this upper bound for $\log \alpha$ into \eqref{eq:n1-bound} and obtain the inequality
\begin{equation*}
	719 n_1< 2.17\cdot 10^{34} \epsilon^{-2} [\log b \log(719 n_1)]^4.
\end{equation*}
Writing $n'=719 n_1$ and applying Lemma \ref{lem:log-sol} gives us an upper bound for $n'$ and in the sequel for $n_1$, namely
\begin{equation*}
	n_1 < 4.83\cdot 10^{32} \frac{(\log b)^4}{\epsilon^2} \log \left[5.56\cdot 10^{36} \frac{(\log b)^4}{\epsilon^2}\right]^4.
\end{equation*}
This concludes the proof of Theorem \ref{th:n1-bound}.

Now let us assume that Assumption A2 holds. Then we put $\epsilon=\frac{1}{2}$ and get
\begin{equation*}
	n_1< 1.94\cdot 10^{33}(\log b)^4\log\left[2.23\cdot 10^{37}(\log b)^4\right]^4,
\end{equation*}
in particular $n_1 < 1.2\cdot 10^{40}$ for $b=2$ (cf.\ Corollary \ref{cor:ex}). If we insert this upper bound into \eqref{eq:alpha-bound-last-case} with $\epsilon=\frac{1}{2}$, then we obtain
\begin{equation*}
	\log \alpha < 9.4\cdot 10^{9}\left[\log\left(1.4\cdot 10^{36}(\log b)^4\log\left[2.23\cdot 10^{37}(\log b)^4\right]^4\right) \log b\right]^2
\end{equation*}
which finally proves Theorem \ref{th:absolute}.

\bibliographystyle{abbrv}
\bibliography{Pillai}
\end{document}